\documentclass{article}
\usepackage{amsfonts}
\usepackage{amsmath}
\usepackage{amssymb}
\usepackage{mathtools}
\usepackage{fullpage}
\usepackage{amsthm}
\usepackage{thmtools}
\usepackage{thm-restate}
\usepackage{graphicx}
\usepackage{relsize}
\usepackage{todonotes}
\usepackage{hyperref}
\usepackage{tikz-cd}
\usepackage{tikz}
\usepackage[capitalise]{cleveref}
\usepackage{placeins}
\usepackage{geometry}

\newtheorem{theorem}{Theorem}[section]
\newtheorem{proposition}[theorem]{Proposition}
\newtheorem{lemma}[theorem]{Lemma} 
\newtheorem{corollary}[theorem]{Corollary}

\theoremstyle{definition}

\newtheorem{definition}[theorem]{Definition}
\newtheorem{notation}[theorem]{Definition}

\newtheorem{innercustomthm}{}
\newenvironment{customthm}[1]
{\renewcommand\theinnercustomthm{#1}\innercustomthm}
{\endinnercustomthm}

\newcommand{\N}{{\mathbb N}}
\newcommand{\R}{{\mathbb R}}
\renewcommand{\S}{{\mathbb S}}
\newcommand{\Q}{{\mathbb Q}}
\newcommand{\Z}{{\mathbb Z}}
\newcommand{\CZ}{{\mathcal Z}}
\newcommand{\CC}{{\mathcal C}}

\newcommand{\Lk}{\operatorname{Lk}}

\newcommand{\Vertices}{\operatorname{Vertices}}

\title{Constructing groups of type $FP_2$ over fields but not over the integers}

\author{Robert Kropholler}

\begin{document}
\maketitle

\begin{abstract}
	We construct examples of groups that are $FP_2(\Q)$ and $FP_2(\Z/p\Z)$ for all primes $p$ but not of type $FP_2(\Z)$. 
\end{abstract}

\section{Introduction}

We begin with a definition:
\begin{definition}
	A group $G$ is of {\em type $FP_n(R)$} if there is an exact sequence: 
	$$P_n\to\dots\to P_2\to P_1\to P_0\to R\to 0$$
	of projective $R G$-modules such that $P_i$ is finitely generated and $R$ is the trivial $R G$-module. 
\end{definition}

Using the chain complex of the universal cover of a presentation 2-complex we see that finitely presented groups are of type $FP_2(\Z)$. 
Moreover, if a group is of type $FP_2(\Z)$, then it is of type $FP_2(R)$ for any ring $R$. 

In \cite{BeBr}, the first examples of groups that are of type $FP_2(\Z)$ but not finitely presented are given. 
More recently, there have been many new constructions of groups of type $FP_2$ with interesting properties, see \cite{BeBr, Lea, Lea2, KLS, Kropholler, LBrown}.
In particular, there are various constructions of uncountable families of groups of type $FP_2(\Z)$ \cite{Lea, Kropholler, LBrown}. 

It is also possible to use the examples of \cite{BeBr, Lea} to give examples of groups that are of type $FP_2(R)$ but not $FP_2(\Z)$ for certain rings $R$. 
The construction of \cite{BeBr}, takes in a connected flag complex $L$ and constructs a group $BB_L$ that is of type $FP_2(R)$ if and only if $H_1(L; R) = 0$. 
Since these flag complexes are finite, it follows that if $H_1(L; \Z/p\Z) = 0$ for all primes $p$, then $H_1(L; \Z) = 0$. 
Thus, if $BB_L$ is $FP_2(\Z/p\Z)$ for all primes $p$, then it is type $FP_2(\Z)$. 
Similar results can be obtained for the groups constructed in \cite{Lea}. 

In this paper we build on the work of \cite{Lea}. Leary built uncountably many groups of type $FP_2$ by taking branched covers of a cube complex $X$. 
Leary's construction takes as input a flag complex $L$ and a set $S\subset \Z$. It outputs a cube complex $X_L^{(S)}$ with a height function $f^{(S)}$. These have the property that if a vertex has height in $S$, then the ascending and descending links at $v$ are $L$. If the height of a vertex is not in $S$, then the ascending and descending links are $\tilde{L}$, the universal cover of $L$. 

We build on this construction by varying the covers that can be taken at each height. Our construction is the following:
\begin{restatable}{construction}{constrmain}\label{constrmain}
	Let $L$ be a flag complex. Let $\sigma\colon \Z\to \CC$ be a function, where $\CC$ is the collection of normal covers of $L$. 
	Then there is a cube complex $X_L^\sigma$ and a height function $f_\sigma$ such that if $f_{\sigma}(v) = n$, then the ascending and descending links of $v$ are exactly $\sigma(n)$. 
\end{restatable}

This cube complex arises as a branched cover and there is a group of deck tranformations $G_L(\sigma)$. Thus we can use the cube complex $X_L^{\sigma}$ to investigate the finiteness properties of $G_L(\sigma)$. We obtain the following theorem: 

\begin{customthm}{\cref{thm:typefpn}}
	Let $\sigma, L, \CC$ be as above. Suppose that $\pi_1(L)/\pi_1(\sigma(n))$ is of type $FP_k(R)$ for all $n$.  Then $G_L(\sigma)$ is type $FP_k(R)$ if and only if $\tilde{H}_i(\sigma(n); R)$ vanishes for all but finitely pairs $(i, n)$ with $n\in \Z$ and $i < k$. 
	
	Similarly, suppose $\pi_1(L)/\pi_1(\sigma(n))$ is of type $FP(R)$ for all $n$. Then $G_L(\sigma)$ is type $FP(R)$ if and only if $\tilde{H}_i(\sigma(n), R)$ vanishes for all but finitely pairs $(i, n)$ with $n\in \Z$ and $i\in\N$.
\end{customthm}

We can use this to construct new examples of groups of type $FP_2(R)$ over various rings $R$. Here we detail two such constructions. 

As pointed out previously, if $G$ is of type $FP_2(\Z)$, then $G$ is of type $FP_2(R)$ for all rings $R$. 
One may hope that there is a collection of rings $\mathcal{R}$ such that if $G$ is of type $FP_2(R)$ for all $R\in \mathcal{R}$, then $G$ is of type $FP_2(\Z)$. 
One possible candidate is the collection of fields. 
We show that this is not the case, proving the following.

\begin{theorem}
	There are groups that are of type $FP_2(F)$ for all fields $F$,  but not of type $FP_2(\Z)$. 
\end{theorem} 

In fact, it is enough to study $\Q$ and $\Z/p\Z$ for all primes $p$. Since if $P$ is the prime subfield of $F$, then $FP_2(P)$ implies $FP_2(F)$. 
Thus, we prove the following:

\begin{customthm}{\cref{thm:acyclicoverQandP}}
	There exist groups that are of type  $FP_2(\Q)$ and $FP_2(\Z/p\Z)$ for all primes $p$ that are not $FP_2(\Z)$. 
\end{customthm}

Moreover, we are able to prove the above theorem for arbitrary sets of primes. 

\begin{customthm}{\cref{thm:setofprimes}}
	Let $S$ be a set of primes. Then there exists a group which is type $FP_2(\Z/p\Z)$ if and only if $p\notin S$. 
\end{customthm}

We highlight one particularly novel corollary to this theorem:
\begin{corollary}
	There exists a group $G$ that is type $FP_2(\Q)$ but not of type $FP_2(\Z/p\Z)$ for any prime $p$. 
\end{corollary}

The second theorem of this paper concerns the constructions from \cite{Lea, Kropholler}. 
In both papers, uncountably many groups of type $FP_2(\Z)$ are constructed by considering subpresentations of an initial group that is known to be of type $FP_2(\Z)$. 
One may believe removing relators from an almost finitely presented group should result in an almost finitely presented group. 
It is clear that one has to be careful when removing relations from a group. 
Indeed, every $n$ generated group appears as a subpresentation of the trivial group presented as $\langle x_1, \dots, x_n\mid F(x_1, \dots, x_n)\rangle$. 

However, in the examples from \cite{Lea, Kropholler} care is taken when considering subpresentations. 
In both cases a generating set for the relation module of the initial group is retained and this is enough to ensure the resulting group is of type $FP_2(\Z)$.
For finitely presented groups, it is true that there is a finite set of relations, that if retained ensure that the subpresentation gives a finitely presented group. 
 
It is then of interest to know whether if one retains an appropriate finite set of relations, say generators for the relation module, does one retain the property of being of type $FP_2(\Z)$? We show that this is not the case: 
\begin{customthm}{\cref{thm:subpresnotfp2}}
	There exists a presentation $G = \langle X\mid R\rangle$ of a group of type $FP_2(\Z)$ such that for any finite subset $T\subset R$, we can find $S$ with $T\subset S\subset R$ and $\langle X\mid S\rangle$ is not of type $FP_2(R)$ for any ring $R$.
\end{customthm}

{\bf Acknowledgements: } The author is thankful to Kevin Schreve for posing questions leading to this article. The author is also grateful to Peter Kropholler and Ian Leary for helpful conversations. 
The author was funded by the Deutsche Forschungsgemeinschaft (DFG, German Research Foundation) under Germany's Excellence Strategy EXC 2044--390685587, Mathematics M\"unster: Dynamics--Geometry--Structure.

\section{Preliminaries}

\subsection{Flag complexes and spherical doubles}

\begin{definition}
	Let $L$ be flag complex. 
	The \emph{spherical double} of $L$, denoted $\S(L)$, is defined by replacing every simplex of $L$ by an appropriately triangulated sphere of the same dimension, in the following way.

	Let $\{v_1,\dots,v_n\}$ be the vertices of $L$.
	The vertex set of $\S(L)$ is a set $\{v_1^+,v_1^-,\dots,v_n^+,v_n^-\}$.
	Thus, each $0$--simplex $\{v_i\}$ of $L$ corresponds to a $0$--sphere $\{v_i^+,v_i^-\}$ in $\S(L)$.
	If $\tau\subset L$ is an $m$--dimensional simplex of $L$, it can be represented as the join $\tau=\{v_{i_0}\}*\dots*\{v_{i_m}\}$ of a collection of $m+1$ vertices of $L$.
	Let $\S(\tau)$ be the join $\S(\tau)=\{v_{i_0}^+,v_{i_0}^-\}*\dots*\{v_{i_m}^+,v_{i_m}^-\}$.
	We have that, $\S(\tau)$ is homeomorphic to an $m$--dimensional sphere.
	We define $\S(L)=\bigcup_{\tau\subseteq L}\S(\tau)$, where $\tau$ ranges through all simplices of $L$.
\end{definition}

It can be shown that $\S(L)$ is a simplicial complex, which is flag if and only if $L$ is flag \cite[Lemma~5.8]{BeBr}.
Also the map $v_i^t\mapsto v_i$, gives a retraction $\S(L)\to L$. 

The following two results are analogous to Proposition 7.1 and Corollary 7.2 of \cite{Lea}, we include the proofs for completeness. 
\begin{proposition}\label{prop:covers}
	Let $L$ be a flag complex. 
	Let $\bar{L}$ be a cover of $L$.
	Then $\S(\bar{L})$ is a cover of $\S(L)$.  
\end{proposition}
\begin{proof}
	This follows since $\S(\bar{L})$ can be seen as the pullback in the square:
	$$\begin{tikzcd}
	\S(\bar{L})\arrow[r]\arrow[d]& \S(L)\arrow[d, "r"] \\
	\bar{L}\arrow[r]& L
	\end{tikzcd}$$
	The lower map is a covering this, we have a pull back of a covering map which is also a covering map. 
\end{proof}

\begin{corollary}\label{cor:subgroupscovers}
	Let $r\colon \S(L)\to L$ be the retraction above. 
	Then $\pi_1(\S(\bar{L})) = r_*^{-1}(\pi_1(\bar{L}))$.
\end{corollary}
\begin{proof}
	Since the diagram in the proof of \Cref{prop:covers} commutes we see that $\pi_1(\S(\bar{L})) \subset r_*^{-1}(\pi_1(\bar{L}))$.
	
	Let $\gamma$ be a loop in $\S({L})$ such that $r\circ \gamma$ is an element of $\pi_1(\bar{L})$.
	Then we can lift $r\circ\gamma$ to a loop $\gamma'$ in $\bar{L}$. 
	Then $(\gamma, \gamma')$ defines a loop in $\S(\bar{L})$ which maps to $\gamma$. 
	Thus $\pi_1(\S(\bar{L})) \supset r_*^{-1}(\pi_1(\bar{L}))$.
\end{proof}

\subsection{Morse theory}

For full details, we refer the reader to \cite{BeBr}.  

A map $f\colon X\to\R$ defined on a cube complex $X$ is a \emph{Morse function} if 
\begin{itemize}
\item for every cell $e$ of $X$, with characteristic map $\chi_e\colon [0,1]^m\to e$, the composition $f\circ\chi_e\colon [0,1]^m\to\R$ extends to an affine map $\R^m\to \R$ and $f\circ\chi_e$ is constant only when $\dim e=0$;
\item the image of the $0$--skeleton of $X$ is discrete in $\R$.
\end{itemize}

Suppose $X$ is a cube complex, $f\colon X\to \R$ a Morse function.
The {\em ascending link} of a vertex $v$, denoted $\Lk_{\uparrow}(v, X)$ is the subcomplex of $\Lk(v, X)$ corresponding to cubes $C$ such that $f|_C$ attains its minimum at $v$. The {\em descending link}, $\Lk_{\downarrow}(v, X)$, is defined similarly replacing minimum with maximum. 

\subsection{Right-angled Artin and Bestvina-Brady groups}

\begin{definition}
	Let $L$ be a flag complex.
	The {\em right-angled Artin group}, or RAAG, associated to $L$ is given by the presentation:
	\[
	A_L=\langle v\in\Vertices(L)  \mid [a_i,a_j]=1 \text{ if } \{a_i,a_j\} \text{ is an edge of }L\rangle.
	\] 
\end{definition}

Given a right-angled Artin group $A_L$, the {\em Salvetti complex}, $S_L$ associated to $A_L$ is a cube complex $S_L$ defined as follows.
For each $v\in\Vertices(L)$ let $S^1_{v}$ be a circle endowed with a structure of a CW complex having a single $0$--cell and a single $1$--cell.
Let $T=\prod_{v} S^1_{v}$ be an $n$--dimensional torus with the product CW structure.
For every simplex $K\subset L$, define a $k$--dimensional torus $T_K$ as a Cartesian product of CW complexes:
$T_K=\prod_{v\in K} S^1_{v}$
and observe that $T_K$ can be identified as a combinatorial subcomplex of $T$.
Then the {\em Salvetti complex} is 
\[
S_L=\bigcup\big\{T_K\subset T\mid K \text{ is a simplex of }L\big\}.
\]

The link of the single vertex of $S_L$ is $\S(L)$.
This is a flag simplicial complex, and hence $S_L$ is a non-positively curved cube complex. 
It follows that the universal cover $X_L = \tilde S_L$ is a CAT(0) cube complex. 

We can define a homomorphism $\phi\colon A_L\to\Z$ by sending each generator to 1. 
We can realise this topologically as a map $f\colon S_L\to S^1$ by restricting the map $\hat{f}\colon T\to S^1$ given by $(x_1, \dots, x_n)\mapsto x_1 + \dots + x_n$.   
Let $BB_L\vcentcolon = \ker(\phi)$. 
We can lift $f$ to a map which we also call $f\colon X_L/BB_L\to \R$. 
This is a Morse function. 
The ascending and decending links are copies of $L$ spanned by $v_i^+$ and $v_i^-$ respectively.

\section{New sequences of covers}

In \cite{Lea}, branched covers of $X_L/BB_L$ were taken to obtain uncountably many groups $G_L(S)$ depending on $S\subset\Z$, whose finiteness properties are controlled by the topology of $L$.
In this section, we generalise this machinery to construct various new groups.

Throughout, let $L$ be a flag complex. 
Let $\CC$ be the set of normal covers of $L$. 
Let $\CZ = \CC^\Z$ be the set of functions $\sigma\colon\Z\to \CC$.
Define a partial ordering on $\CZ$ by $\sigma \preceq \sigma'$ if $\sigma(n)$ is a cover of $\sigma'(n)$. 

We will associate to each element $\sigma\in \CZ$ a group as follows.

{\bf Construction of $G_L(\sigma)$:} 
For each integer $n$ there is a single vertex $v$ of $X_L/BB_L$ such that $f(v) = n$. 
Let $\gamma_{i, n}$ be a sequence of loops in $\S(L)$ that normally generate $G_n = \pi_1(\S(\sigma(n)))\leq \pi_1(\S(L))$. 
Let $V$ be the vertex set of $X_L/BB_L$ and let $V(\sigma) \vcentcolon = \{v\in X_L/BB_L\mid \sigma(f(v))\neq L\}\subset V$. 
By identifying $\Lk(v, X)$ with the boundary of the $\frac{1}{4}$ neighbourhood of $v$ we can consider $\gamma_{i, n}$ as loops in $X_L/BB_L$. 
Let $Y_L(\sigma)$ be the complex obtained from $(X_L/BB_L)\smallsetminus V(\sigma)$ by attaching disks to all the loops $\gamma_{i, n}$. 
Let $G_L(\sigma) = \pi_1(Y_L(\sigma))$. 

We can also view the group $G_L(\sigma)$ as a group of deck transformations of a branched cover of cube complexes. 
\begin{theorem}
	There is a CAT(0) cube complex $X_L^{\sigma}$ with a branched covering map $b\colon X_L^\sigma\to X_L/BB_L$ such that $G_L(\sigma)$ is the group of deck transformations of this branched cover. 
\end{theorem}
\begin{proof}
	We construct $X_L^\sigma$ as follows. 
	Let $\widetilde{Y_L(\sigma)}$ be the universal cover of $Y_L(\sigma)$. 
	Let $\mathcal{D}$ be the collection of open disks added to $(X_L/BB_L)\smallsetminus V(\sigma)$ to obtain $Y_L(\sigma)$ and let $\tilde{\mathcal{D}}$ be the collection of lifts of $\mathcal{D}$ to $\widetilde{Y_L(\sigma)}$. 
	Let $Z_L(\sigma)$ be $\widetilde{Y_L(\sigma)}\smallsetminus\tilde{\mathcal{D}}$. 
	Then the covering map $p\colon \widetilde{Y_L(\sigma)}\to Y_L(\sigma)$ restricts to a covering map $Z_L\to (X_L/BB_L)\smallsetminus V(\sigma)$. 
	We can now lift the metric and complete to obtain a branched cover $b\colon X_L^\sigma\to X_L/BB_L$.
	The deck group is exactly the deck group of the covering $\widetilde{Y_L(\sigma)}\to Y_L(\sigma)$, this is  $G_L(\sigma)$. 
	
	Given a vertex $v\in X_L^\sigma$ we obtain a covering map $b_v\colon \Lk(v, X_L^\sigma)\to \Lk(b(v), X_L)$. 
	Since the cover of a flag complex is a flag complex we see that $X_L^\sigma$ is non-positively curved. 
	
	Taking the completion adds in the missing vertices of $Z_L$. 
	The vertices added cone off their links. 
	As such the boundary of each disk in $\tilde{\mathcal{D}}$ is trivial in $X_L^\sigma$ and thus $\pi_1(\widetilde{Y_L(\sigma)})$ surjects $\pi_1(X_L^\sigma)$. 
	We conclude, $X_L^\sigma$ is simply connected.
	Thus, $X_L^{\sigma}$ is non-positively curved and simply connected and hence CAT(0). 
\end{proof}

There is a Morse function $f_\sigma\colon X_L^\sigma\to \R$ given by composition $f\circ b$. 
Since $G_L(\sigma)$ is the covering group of $Z_L\to X_L\smallsetminus V(X_L)$, we see that it acts on $X_L^\sigma$ cellularly and freely away from the vertex set. 
Moreover, $G_L(\sigma)$ acts properly, freely and cocompactly on $f_\sigma^{-1}(\frac{1}{2})$. 
Thus by understanding this level set we can understand finiteness properties of $G_L(\sigma)$. 
We will proceed by understanding the ascending and descending links of the Morse function $f_\sigma$. 

\begin{lemma}
	Let $v$ be a vertex of $X_L^\sigma$ such that $f_\sigma(v) = n$. 
	Then $\Lk(v, X_L^\sigma) = \S(\sigma(n))$ and 	$\Lk_\uparrow(v, X_L^\sigma) = \Lk_\downarrow(v, X_L^\sigma) = \sigma(n)$. 
\end{lemma}
\begin{proof}
	Let $\mathcal{D}_v$ be the collection of disks glued at $v$ in $\widetilde{Y_L(\sigma)}$. 
	There is a retraction $X_L^\sigma\smallsetminus \{v\}\to N(v, X_L^\sigma)\smallsetminus \{v\}$ where $N(v, X_L^\sigma)$ is a neighbourhood of $v$. 
	Let $\gamma$ be the attaching map of a disk in $\tilde{\mathcal{D}}\smallsetminus\mathcal{D}_v$, then this bounds a disk in $X_L^\sigma\smallsetminus \{v\}$. 
	Thus we can extend the retraction over elements of $\tilde{\mathcal{D}}\smallsetminus\mathcal{D}_v$. 
	This gives a retraction $\widetilde{Y_L(\sigma)}\to Y_L(\sigma, v)$, where $Y_L(\sigma, v) = (N(v, X_L^\sigma)\smallsetminus \{v\})\cup \mathcal{D}_v$
	Since the former is simply connected, so is the latter. 
	
	We see that $Y_L(\sigma, v)$ is homotopy equivalent to a cover $\overline{\S({L})}$ of $\S(L)$ together with disks glued to each lift of $\gamma_{i, n}$. 
	Since this is simply connected we see that $\overline{\S({L})}$ is the cover corresponding $\langle \langle\gamma_{i, n}\rangle\rangle$. 
	Since $\{ \gamma_{i, n}\}$ normally generates $\pi_1(\S(\sigma(n)))$ we see that $\overline{\S({L})} = \S(\sigma(n))$ and the ascending link is the preimage of $L$ which is exactly $\sigma(n)$. 
	Similarly the descending link is $\sigma(n)$. 
\end{proof}

This allows us to understand the finiteness properties of $G_L(\sigma)$. 

Firstly, we recall a simplified version of Brown's criterion \cite{BrownCrit} (from \cite{Lea}) for a group to be of type $FP_k(R)$. 
\begin{theorem}\label{thm:brown}
	Suppose that $X$ is a finite-dimensional $R$-acyclic G-CW-complex, and that $G$ acts freely except possibly that some vertices have isotropy subgroups that are of type $FP(R)$ (resp. $FP_k(R)$).  Suppose also that $X = \cup_{m\in \N} X(m)$ where $X(m)\subset X(m+1)\subset\dots\subset X$ is an ascending sequence of $G$-subcomplexes, each of which contains only finitely many orbits of cells. In this case $G$ is $FP(R)$ (resp. $FP_k(R)$) if and only if for all $i$ (resp. for all $i < k$) the sequence $\tilde{H}_i(X(m); R)$ of reduced homology groups is essentially trivial.
\end{theorem}

\begin{theorem}\label{thm:typefpn}
	Let $\sigma, L, \CC$ be as above. Suppose that $\pi_1(L)/\pi_1(\sigma(n))$ is of type $FP_k(R)$ for all $n$.  Then $G_L(\sigma)$ is type $FP_k(R)$ if and only if $\tilde{H}_i(\sigma(n), R)$ vanishes for all but finitely pairs $(i, n)$ with $n\in \Z$ and $i < k$. 
	
	Similarly, suppose $\pi_1(L)/\pi_1(\sigma(n))$ is of type $FP(R)$ for all $n$. Then $G_L(\sigma)$ is type $FP(R)$ if and only if $\tilde{H}_i(\sigma(n), R)$ vanishes for all but finitely pairs $(i, n)$ with $n\in \Z$ and $i\in\N$.
\end{theorem}
\begin{proof}
	We focus on the proof for $FP_k(R)$, the proof for $FP(R)$ is similar. 

	The group $G_L(\sigma)$ acts on $X_L^\sigma$ freely away from vertices. 
	For vertices at height $n$, the stabiliser is the deck group of the covering $\sigma(n)\to L$. 
	This is exactly $\pi_1(L)/\pi_1(\sigma(n))$ which is of type $FP_k(R)$. 
	
	Let $X(m) = f_\sigma^{-1}([-m-\frac{1}{2}, m + \frac{1}{2}])$. 
	We are now in the situation of \cref{thm:brown}, thus $G_L(\sigma)$ is of type $FP_k(R)$ if and only if for all $i<k$ the sequence $\tilde{H}_i(X(m); R)$ of reduced homology groups is essentially trivial.
	
	If $\tilde{H}_i(\sigma(n), R)$ vanishes for all but finitely pairs $(i, n)$ with $n\in \Z$ and $i < k$, then by \cite[Corollary 2.6]{BeBr} we can find an $l$ such that for all $m > l$ the inclusion $X(l) \to X(m)$ induces an isomorphism on all $H_i$ for $i< k$. 
	Since homology commutes with direct limits we see that $H_i(X(l)) = H_i(X_L^\sigma) = 0$ for all $i<k$. 
	Thus the system is essentially trivial. 
	
	Now conversely suppose that there are infinitely many $n$ such that $H_i(\sigma(n), R)\neq 0$ for some $i<k$. 
	Thus for each $l>0$ we can find $l' > l + 1$ and $v$ such that $H_i(\sigma(l'); R)$ or $H_i(\sigma(-l'); R)$ is non-trivial for some $i<k$. 
	We will assume that $\sigma(l')$ has the non-trivial homology group. 
	Let $v$ be a vertex at height $l'$. 
	There is a map from $X(l'-1)$ to $\Lk(v, X_L^\sigma)$. 
	We can further compose with the retraction $\Lk(v, X_L^\sigma)\to \sigma(l')$. 
	By extending geodesics from $v$ downwards, we can view $\sigma(l')$ as a subspace of $X(l'-1)$ and this composition will be a retraction. 
	However, the inclusion $\sigma(l')\to X(l')$ induces the trivial map on homology, thus we have an element in the kernel of the map $H_i(X(l'-1); R) \to H_i(X(l'); R)$. 
	Thus the system of homology groups is not essentially trivial and $G_L(\sigma)$ is not of type $FP_k(R)$. 
\end{proof}

We can also prove similar results about finite presentability of $G_L(\sigma)$. 
\begin{theorem}
	Suppose that $\pi_1(L)/\pi_1(\sigma(n))$ is finitely presented for all $n$. 
	Then $G_L(\sigma)$ is finitely presented if and only if $\pi_1(\sigma(n))$ vanishes for all but finitely many $n$. 
\end{theorem}
\begin{proof}
	For one direction, suppose that $G_L(\sigma)$ is finitely presented. 
	Since it acts freely and cocompactly on $f_\sigma^{-1}(\frac{1}{2})$, there are finitely many orbits of loops which normally generate $\pi_1(f_\sigma^{-1}(\frac{1}{2}))$. 
	Since the limit of $X(m)$ is simply connected, we see that there is an $l$ such that each of these loops is trivial in $X(l)$. 
	Thus the inclusion $f_\sigma^{-1}(\frac{1}{2})\to X(l)$ is trivial on fundamental groups. 
	By \cite[Corollary 2.6]{BeBr}, we have that $f_\sigma^{-1}(\frac{1}{2})\to X(l)$ is also a surjection. 
	Thus, $X(l)$ is simply connected. 
	Now using the retraction from the proof of \cref{thm:typefpn} we obtain a $\pi_1$-surjective map $X(l)\to \sigma(m)$ for all $m$ such that $|m|>l$.
	Thus, we see that for all $m$ such that $|m|>l$, we must have $\sigma(m)$ is simply connected. 
	
	For the other direction, suppose that there is an $l$ such that if $|m|>l$ we have that $\sigma(m)$ is simply connected. 
	Then, by \cite[Corollary 2.6]{BeBr}, $X(l)$ is simply connected and has a cocompact action by $G_L(\sigma)$. 
	We now have a cocompact action of $G_L(\sigma)$ on a simply connected CW complex where stabilisers of cells are finitely presented. 
	Thus $G_L(\sigma)$ is finitely presented. 
\end{proof}

\section{Presentations for $G_L(\sigma)$}

This section closely follows Section 14 of \cite{Lea}.

We begin by describing presentations of the groups $G_L(\sigma)$ obtained in the previous section. 
\begin{notation}
	Let $L$ be a simplicial complex. 
	Let $E$ be the set of edges of $L$. 
	For a loop $c = (e_1, \dots, e_l)$ in $L$. Let $c^{[k]}$ denote the word $e_1^ke_2^k\dots e_l^k$ in $F(E)$. 
\end{notation}

\begin{theorem}\label{thm:presentationforgroups}
	Let $\gamma_{i, n}$ be a collection of loops that normally generate $\pi_1(\sigma(n))$. 
	Let $E$ be the edges of $L$. 
	Suppose that $\sigma(0) = L$. 
	Then $G_L(\sigma)$ has the following presentation:
	$$\langle E\mid efg, gfe \text{ for each triangle $e, f, g$ in $L$ }, \gamma_{i, n}^{[n]} \text{ for $n\in \Z$ and all } i \rangle. $$
\end{theorem}
\begin{proof}
	Recall that we have a Morse function $f\colon X_L/BB_L\to \R$. 
	Let $Y_0 = f^{-1}(0)$. 
	Since $\sigma(0) = L$ we have that $Y_0\subset Y_L(\sigma)$
	By \cite[Corollary 10.4]{Lea}, we have that the inclusion $Y_0\to Y_L(\sigma)$ is a surjection on $\pi_1$. 
	From \cite[Theorem 14.1]{Lea}, we have a presentation for $\pi_1(Y_0)$ which is exactly, 
	$$\langle E\mid efg, gfe \text{ for each triangle $e, f, g$ in $L$ }\rangle. $$
	Let $Z = X_L\smallsetminus\{v\in X_L^{(0)}\mid f(v)\neq 0\}$. 
	Then $\pi_1(Z) = \pi_1(Y_0)$. 
	
	Thus to obtain a presentation of $G_L(\sigma)$ we add in the relations coming from the disks in $Y_L(\sigma)$. 
	By \cite[Lemma 14.3]{Lea}, we see that if $D$ is the disk glued to the word $\gamma_{i, n}$, then this corresponds exactly to the relation $\gamma_{i, n}^{[n]}$. 
	Thus we arrive at the desired presentation. 
\end{proof}

Later we will show that there are uncountably many such groups. 
For this purpose it will be useful to know the following: 

\begin{lemma}\label{lem:trivial}
	Let $\gamma$ be an edge loop in $L$. 
	Then $\gamma^{[n]}$ is trivial in $G_L(\sigma)$ if and only if $\gamma$ lifts to a loop in $\sigma(n)$.  
\end{lemma} 
\begin{proof}
	The loop $\gamma^{[n]}$ in $Y_L(\sigma)$ is homotopic to the loop $\gamma$ in the link of the vertex $v$ of $X_L$ at height $n$.
	We will label this $\gamma_n$.
	Note that $\gamma_n$ belongs to the ascending or descending link of $v$ (depending on the sign of $n$) and this subspace is exactly $L$. 
	
	If $\gamma$ lifts to a loop in $\sigma(n)$, then $\gamma$ defines an element of $\pi_1(\sigma(n))$ and is homotopic to a product of the loops $\gamma_{i, n}$. 
	Thus together with the triangle relations we see that $\gamma^{[n]}$ is trivial in $\pi_1(Y_L(\sigma))$. 
	
	Now suppose that $\gamma^{[n]}$ is trivial in $\pi_1(Y_L(\sigma))$, then the loop $\gamma^{[n]}$ lifts to a loop in $\widetilde{Y_L(\sigma)}$. 
	We can now lift the homotopy and see that $\gamma_n$ also lifts to a loop in $\widetilde{Y_L(\sigma)}$ which we will call $\lambda_n$. 
	This loop must be in the ascending or descending link of some vertex $w$ of $\widetilde{Y_L(\sigma)}$. 
	However, this link is exactly $\sigma(n)$, thus $\lambda_n$ is a loop in $\sigma(n)$ which is a lift of $\gamma$. 
\end{proof}

\section{Maps between $G_L(\sigma)$} 

Recall that we partially order the set $\CZ$ by $\sigma\preceq\sigma'$ if $\sigma(n)$ is a cover of $\sigma'(n)$. 
Let $\gamma_{i, n}$ be the set of loops that generate $\pi_1(\sigma(n))$ and $\gamma_{i, n}'$ be the set of loops that normally generate $\pi_1(\sigma'(n))$. 
In the case that $\sigma(n)$ is a cover of $\sigma'(n)$ we can assume that $\{\gamma_{i, n}\}\subset\{\gamma_{i, n}'\}$. 
Thus, one can see from the presentations in \Cref{thm:presentationforgroups} that there are surjective maps $G_L(\sigma)\to G_L(\sigma')$ whenever, $\sigma\preceq\sigma'$. 
This map can also be realised at the level of cube complexes. 
\begin{theorem}\label{thm:covers}
	Let $\sigma, \sigma'\in \CZ$. 
	Suppose that $\sigma(n)\preceq\sigma'(n)$. 
	Then there is a branched cover of cube complexes $X_L^\sigma\to X_L^{\sigma'}$ which preserves level sets. 
\end{theorem}
\begin{proof}
	To see this recall $Y_L(\sigma)$ is obtained from $X_L$ by removing all vertices and gluing disks in the link at level $n$ to generators for $\pi_1(\S(\sigma(n)))$. 
	Thus take as our generating set for $\pi_1(\S(\sigma(n)))$ a generating set for $\pi_1(\S(\sigma'(n)))$ along with extra generators. 
	This way we obtain an inclusion $Y_L(\sigma)\to Y_L(\sigma')$. 
	
	Let $\mathcal{D}$ be the set of disks added to obtain $Y_L(\sigma)$. 
	Let $\mathcal{D}'$ be the set of disks added to $Y_L(\sigma)$ to obtain $Y_L(\sigma')$. 
	
	Let $\widetilde{Y_L(\sigma')}$ be the universal cover of $Y_L(\sigma')$. 
	Let $Y_1$ be the space obtained from $\widetilde{Y_L(\sigma')}$ by removing the lifts of disks in $\mathcal{D}'$. 
	Then $Y_1$ is the cover of $Y_L(\sigma)$ corresponding to the kernel of the surjection $G_L(\sigma)\to G_L(\sigma')$. 
	Thus, the universal cover of $Y_1$ is the universal cover of $Y_L(\sigma)$. 
	When we remove the disks and complete we get a branched cover of cube complexes ${X_L^\sigma}\to {X_L^{\sigma'}}$. 
\end{proof}

\begin{corollary}
	Let $\sigma, \sigma'\in\CZ$ with $\sigma\prec\sigma'$. 
	Suppose that $\sigma(0) = \sigma'(0) = L$. 
	Then the Cayley graph of $G_L(\sigma)$ is a cover of the Cayley graph for $G_L(\sigma')$, where both groups have as generating sets the edges of $L$. 
\end{corollary}
\begin{proof}
	In the case that $\sigma(0) = L$ we have that $G_L(\sigma)$ acts freely on $f_{\sigma}^{-1}(0)$. Moreover, it acts transitively on vertices. Thus, the 1-skeleton of $f_{\sigma}^{-1}(0)$ is the Cayley graph for $G_L(\sigma)$. 
	Similar statements hold for $G_L(\sigma')$. 
	
	Now \cref{thm:covers}, we have a covering map $X_L^\sigma\to X_L(\sigma')$ which preserves level sets. Thus we get a covering of Cayley graphs. 
\end{proof}

\begin{theorem}\label{thm:subpresnotfp2}
	There exists a presentation $G = \langle X\mid R\rangle$ of a group of type $FP_2(\Z)$ such that for any finite subset $T\subset R$, we can find $S$ such that $T\subset S\subset R$ and $\langle X\mid S\rangle$ is not of type $FP_2(R)$ for any ring $R$.
\end{theorem}
\begin{proof}
	Throughout the proof let $G = \Z^3\rtimes SL_3(\Z)$ and $N = \Z^3\triangleleft G$. 
	There are two properties of this pair we will use namely, that $G$ is perfect (see for instance \cite{condersl3}) and so $H_1(G; R) = 0$ for all rings $R$ and $H_1(N; R) = R^3$ for all rings $R$. 
	
	Let $L$ be a flag complex with no local cut points and fundamental group $G$. 
	Let $K$ be the cover corresponding to $N$. 
	Let $E$ be the set of edges of $L$. 
	Let $\alpha_i = (e_{1,i}, \dots, e_{n_i,i})$ be a sequence of loops in $L$ that generate $N$. 
	Let $\beta_i = (f_{1, i}, \dots, f_{m_i, i})$ be a sequence of loops in $L$ that generate $G$.
	We can obtain the following presentation for $BB_L$ from \cite{Dicks}:
	$$\langle E\mid efg, gfe \text{ for each triangle $e, f, g$ in $L$ }, \alpha_i^{[n]}, \beta_i^{[n]} \text{ for $n\in \Z$ and all } i \rangle.$$
	Since $G$ is perfect, we obtain from \cite{BeBr} that $BB_L$ is of type $FP_2(\Z)$. 
	
	Let $T$ be a finite subset of the relations of $BB_L$. 
	Let $F = \{n\mid \exists i$ such that $\beta_i^{[n]}\in T\}$. 
	Let $S$ be the the union of the following sets of relations: 
	\begin{itemize}
		\item $T$, 
		\item all triangle relations, 
		\item $\alpha_i^{[n]}$ for all $n\in \Z$, 
		\item $\beta_i^{[n]}$ for $n\in F$
	\end{itemize}
	
	Consider the subpresentation $$\langle E\mid S\rangle$$
	This is a presentation of $G_L(\sigma)$, where 
	$$\sigma(n) = 
	\begin{cases}
	K, &\text{ if } n\notin F, \\
	L, &\text{ if } n\in F. 		
	\end{cases}$$
	Since $F$ is a finite set there are infinitely many vertices such that the ascending and descending link have non-trivial first homology with coefficients in $R$. 
	Thus $G_L(\sigma)$ is not of type $FP_2(R)$ by \cref{thm:typefpn}. 
\end{proof}

\section{Groups that are $FP_2$ over fields}

We are now ready to prove the following theorem. 
\begin{theorem}\label{thm:acyclicoverQandP}
	There exists groups that are of type $FP_2(\Q)$ and $FP_2(\Z/p\Z)$ for all $p$ but not of type $FP_2(\Z)$. 
\end{theorem}
\begin{proof}
	To do this we find a finitely presented group $G$ with a sequence of subgroups $G_n$ such that 
	$$H_1(G_n ; \Z) = \begin{cases}
		\Z_n^m,  &\text{ if  $n$ is prime,}\\
		0 , &\text{ if $n$ is not prime.}
	\end{cases}$$

	Let $G = SL_3(\Z)$. 
	Let $G_p$ be the level $p$ congruence subgroup. 
	By \cite{Lee}, we have that $H_1(G_p ; \Z) = \Z/p\Z^8$. 
	
	Now let $L$ be a flag complex with no local cut points with fundamental group $G$. 
	Let $G_n$ be the level $n$ congruence subgroup if $n > 2$ is prime and the trivial subgroup otherwise. 
	Let $L_n$ be the cover corresponding to $G_n$. 
	Let $\sigma$ be the function assigning $n$ to $L_n$. 
	
	In this case $L_n$ is a trivial or finite cover of $L$. 
	In either case, the quotient is finitely presented and hence of type $FP_2$ over any ring. 
	
	Since all the homology groups considered are finite we see that $H_1(\sigma(n); \Q)$ vanishes for all $n$. 
	Also $H_1(\sigma(n); \Z/p\Z)$ is non-trivial if and only if $n = p$. 
	Thus we can apply \cref{thm:typefpn} to see that $G_L(\sigma)$ is of type $FP_2(\Q)$ and $FP_2(\Z/p\Z)$ for all $p$. 
	
	However, there are infinitely many $n$ with $H_1(\sigma(n); \Z)$ non-trivial. Thus, by \Cref{thm:typefpn} we see that $G_L(\sigma)$ is not of type $FP_2(\Z)$. 
\end{proof}

One would imagine that it is possible to prove the corresponding result for $FP_k$ or even $FP$. 
The above theorem gives a template for how to do this. 
To prove the above theorem for type $FP_k$ one would need a flag complex $L$ and a sequence of normal covers $L_n$ satisfying the following conditions: 
\begin{itemize}
	\item For infinitely many $n$, there is an $i < k$ such that $H_i(L_n; \Z)$ does not vanish.  
	\item For all but finitely many pairs $(i, n)$ with $i<k$, we have that $H_i(L_n; \Q)$  vanishes. 
	\item For each prime $p$, we have that for all but finitely many pairs $(i, n)$, with $i<k$, we have that $H_i(L_n; \Z/p\Z)$ vanishes. 
	\item For all $n$ and all $p$ the quotient $\pi_1(L)/\pi_1(L_n)$ is of type $FP_k(\Q)$ and $FP_k(\Z/p\Z)$. 
\end{itemize}
For the $FP$ result replace $FP_k$ by $FP$ and remove the $i<k$ assumption throughout. 

In a similar way to \cref{thm:acyclicoverQandP}, we can prove the following theorem.

\begin{theorem}\label{thm:setofprimes}
	Let $\mathcal{P}$ be the set of primes. 
	For each subset $S$ of $\mathcal{P}$ there is a group which is type $FP_2(\Z/p\Z)$ if and only if $p\notin S$. 
	
	Moreover, we can construct such a group that has a proper action on a 3-dimensional CAT(0) cube complex.
\end{theorem}
\begin{proof}
	Let $L$ be a flag complex with fundamental group $GL_3(\Z)$. 
	For $p>2$, let $G_p$ denote the level $p$ congruence subgroup in $SL_3(\Z)$. 
	Note that $G_p$ is still finite index and normal in $GL_3(\Z)$. 
	Let $G_2 = GL_3(\Z)$. 
	Let $\bar{L}$ be the cover of $L$ corresponding to $SL_3(\Z)$. 
	Let $L_{p}$ be the cover of $L$ corresponding to $G_p$. 
	Thus $H_1(L_p; \Z)$ is a finite $p$-group.
	
	Let $S = \{p_1, p_2, p_3, \dots\}$
	Let $(a_n)_{n\in \N}$ be any sequence which contains $p_i$ infinitely many times for each $p_i\in S$. 
	For instance, we can take the sequence 
	$$p_1, p_1, p_2, p_1, p_2, p_3, p_1, p_2, p_3, p_4, p_1, \dots.$$
	
	Define $\sigma\colon \Z\to \CC$ as follows
	$$\sigma(n) = 
	\begin{cases}
	\bar{L}, &\text{if $n<0$},\\
	L_{a_n}, &\text{if $n\geq 0$}.
	\end{cases}$$
	Since each $p_i$ appears infinitely many times in $(a_n)$ we have by \cref{thm:typefpn}, that $G_L(\sigma)$ is $FP_2(\Z/p\Z)$ if and only if $p\notin S$. 
	
	Since all the covers taken were finite index we see that all the vertex stabilisers are finite. Thus the action of $G_L(\sigma)$ on $X_L^\sigma$ is proper. 
\end{proof}

\section{Uncountably many quasi-isometry classes}

We can show that there are uncountably many quasi-isometry classes of groups as in \cref{thm:setofprimes}. 
The proof given here closely follows that of \cite{KLS}.
The idea of the proof is to interleave the sequence of covers from \cref{thm:setofprimes}, with the sequences of universal covers as in \cite{KLS}. 
Thus, we can use the sequence from covers from \cref{thm:setofprimes} to obtain the desired finiteness properties and we use the relations (or lack thereof) from the universal covers to obtain uncountably many quasi-isometry classes.

Let $\sigma, \sigma'\in \CZ$ with $\sigma\prec\sigma'$. 
Let $M(\sigma, \sigma') = \min\{|n| \mid \sigma(n)\neq \sigma'(n)\}$. 

\begin{lemma}\label{lem:minlength}
	Suppose that $L$ is $d$-dimensional and $\sigma, \sigma'\in \CZ$ with $\sigma\prec\sigma'$ and $\sigma(0) = \sigma'(0) = L$, and take the standard generating set for $G_L(\sigma)$ and $G_L(\sigma')$. The word length of any non-identity element in the kernel of the map $G_L(\sigma)\to G_L(\sigma')$ is at least $M(\sigma, \sigma')\sqrt{2/d+1}.$
\end{lemma}
\begin{proof}
	This follows from Lemmas 3.1 and 3.2 of \cite{KLS}.
\end{proof}

We will use the taut loop length spectrum of Bowditch \cite{bowditch}. 

\begin{definition}
	Let $\Gamma$ be a graph and $l\in \N$. 
	Let $\Gamma_l$ denote the 2-complex with 1-skeleton $\Gamma$ and a 2-cell attached to each loop of length $<l$. 
	An edge loop of length $l$ is {\em taut} if it is not null-homotopic in $\Gamma_l$. 
	Bowditch's {\em taut loop length spectrum}, $H(\Gamma)$ is the set of lengths of taut loops.
\end{definition}

\begin{definition}
	Let $H,  H'$ be two sets of natural numbers. We say that $H$ and $H'$ are {\em $k$-related} if for all $l\geq k^2 +2k +2 $, whenever $l\in H$ then there is some $l'\in H'$ such that $\frac{l}{k}\leq l'\leq lk$ and vice versa. 
\end{definition}

The key element from \cite{bowditch} is the following relating the taut loop length spectrum to quasi-isometries. 

\begin{lemma}
	If (connected) graphs $\Gamma$ and $\Lambda$ are $k$-quasi-isometric, then $H(\Gamma)$ and $H(\Lambda)$ are $k$-related.
\end{lemma} 

We are now ready to prove the main theorem of this section. 
The proof is similar to that of Theorem 5.2 in \cite{KLS}. 
Here is a brief outline.
	For each $F\subset \N$ we will construct a function $\sigma_F$. 
	We will then give a rough computation of the taut loop length spectrum for the Cayley graph for $G_L(\sigma_F)$. 
	This will show that if $G_L(\sigma_F)$ is quasi-isometric to $G_L(\sigma_{F'})$, then $F\triangle F'$ is finite and conclude the desired result. 
	The key change from \cite{KLS} is that due to the sequence of covers we cannot take a single constant $C$ and must take a sequence of constants $C_i$ satisfying certain conditions relating to the function $\sigma_F$.  
Comparing to the proof of Theorem 5.2 in \cite{KLS} we have replaced the constant $\beta$ by the sequence of constants $r_p$ and chosen $C_i$ to ensure the arguments contained there still work.

\begin{theorem}\label{thm:uncountable}
	Let $S$ be a set of primes. 
	Then there are uncountably many quasi-isometry classes of groups that are of type $FP_2(\Z/p\Z)$ if and only if $p\in S$. 
\end{theorem}
\begin{proof}
	Let $L$ be a flag complex with fundamental group $GL_3(\Z)$. 
	For each prime $p$, let $L_p$ be the cover of $L$ corresponding to the level-$p$ congruence subgroup in $SL_3(\Z)$. 
	Let $L_1$ be the cover corresponding to $SL_3(\Z)$. 
	Let $L_0$ be the universal cover $\tilde{L}$ of $L$. 
	
	Given a finite set $\Omega$ of loops that normally generate $\pi_1(L_p)$, define $r_p(\Omega)$ as the maximum length of a loop in $\Omega$. 
	Define $r_p$ to be the minimum of $r_p(\Omega)$ as $\Omega$ runs over all possible normal generating sets for $\pi_1(L_p)$. 
	
	Let $(b_n)$ be the following sequence:
	$$b_{n} = 
	\begin{cases}
	a_m, & \text{if $n = 2m+1$}, \\
	1, & \text{if $n = 2m$},\\
	\end{cases}$$ 
	
	Let $\alpha = \sqrt{2/(d+1)}$, where $d$ is the dimension of $L$.
	Let $C_n$ be a sequence of integers satisfying the following conditions: 
	\begin{itemize}
		\item $C_1\alpha > 3$, 
		\item $C_n\alpha > r_{b_{n-1}}$,
		\item $C_n\alpha > r_{b_n}$ 
		\item $C_n >C_{n-1}$. 
	\end{itemize}
	From this we can deduce that $\alpha C_n^{2^n} > r_{b_{n-1}}C_{n-1}^{2^{n-1}}$. 
	
	Let $F\subset \N$.
	Let $\overline{F} = \{2n\mid n\in F\}\cup\{2n+1\mid n\in \N\}$.  
	Define $\sigma_F$ as follows:
	$$\sigma_F(n) = \begin{cases}
	L_{b_i}, &\text{if $n = {C_i^{2^{i}}}$ and $i\in \overline{F}$,}\\
	L, &\text{if $n = 0$,}\\
	{L_0}, &\text{otherwise.}
	\end{cases}$$
	
	Let $\Gamma_F$ be the Cayley graph of $G_L(\sigma_F)$ with generating set the edges of $L$. 
	We will prove the following two propositions about $H(\Gamma_F)$:
	
	\begin{itemize}
		\item If $k\in \overline{F}$, then $H(\Gamma_F)\cap [C_k^{2^k}\alpha, C_k^{2^k}r_{b_k}] \neq \emptyset$. 
		\item If $k>3$ and $k \notin \cup_{l\in \overline{F}}[C_l^{2^l}\alpha, C_l^{2^l}r_{b_l}]$, then $k\notin H(\Gamma_F)$. 
	\end{itemize}
	
	Suppose $k\in \overline{F}$, let $F' = F\smallsetminus\{k\}$. 
	There is a surjection $G_L(\sigma_F)\to G_L(\sigma_{F'})$. 
	Let $K$ be the kernel of this surjection. 
	By \cref{lem:minlength}, any non-identity element of $K$ has length at least $\alpha C_k^{2^k}$. 
	We also know that there is an element of length $r_kC_k^{2^k}$ in $K$. 
	The length of the shortest element of $K$ defines a member of $H(\Gamma_F)$. 
	Thus we obtain the first statement. 
	
	For the second statement, let $k\notin F$. 
	Choose $n\in \N$ maximal such that $r_{b_n}C_n^{2^n} < k$ or $-1$ otherwise.
	Let $F' = F\cap [0, n]$. 
	Let $K$ be the kernel of the map $G_L(\sigma_{F'})\to G_L(\sigma_F)$. 
	Consider the covering map $\Gamma_{F'}\to \Gamma_F$ coming from \cref{thm:covers}. 
	Every relator in $G_L(\sigma_{F'})$ has length $\leq r_{b_n}C_n^{2^n}$. 
	
	Now suppose that $\gamma$ is a loop of length $k$ in $\Gamma_F$. 
	We can lift $\gamma$ to $\Gamma_{F'}$. 
	If $\gamma$ lifts to a loop in $\Gamma_{F'}$, then it must be a consequence of loops of length $\leq r_{b_n}C_n^{2^n}$. 
	Thus, $\gamma$ cannot be taut. 
	
	Now suppose that $\gamma$ lifts to a non-closed path. 
	In this case $\gamma$ defines an element of $K$. 
	The shortest such element has length $\geq \alpha C_m^{2^m}$ where $m = M(\sigma_F, \sigma_{F'})$. 
	By choice of $n$ we have that $k\leq r_mC_m^{2^m}$. 
	Thus, we obtain that $k\in [\alpha C_m^{2^m}, r_{b_m}C_m^{2^m}]$. 
	However, this contradicts the choice of $k$ and thus $\gamma$ is not taut. 
	
	To complete the proof,  suppose that $l\in [\alpha C_m^{2^m}, r_{b_m}C_m^{2^m}]$ and $l'\in [\alpha C_{m+n}^{2^{m+n}}, r_{b_{m+n}}C_{m+n}^{2^{m+n}}]$ for some $n > 0$. 
	Then $$\frac{l'}{l} \geq \frac{\alpha C_{m+n}^{2^{m+n}}}{r_{b_m}C_m^{2^m}} > \frac{\alpha C_m^{2^{m+n}}}{r_{b_m}C_m^{2^m}} \geq \frac{\alpha C_m^{2^{m} - 1}} {r_{b_m}} > C_m^{2^{m}-2}.$$
	
	Suppose that $H(\Gamma_F)$ and $H(\Gamma_{F'})$ are $k$-related. If $n\in F\triangle F'$, then $C_n^{2^{2n}-2}\leq k$. 
	Thus if $F\triangle F'$ is infinite, then $H(\Gamma_F)$ and $H(\Gamma_{F'})$ are not $k$-related. 
	Hence, $G_L(\sigma_F)$ is not quasi-isometric to $G_L(\sigma_{F'})$. 
\end{proof}

While we have used the framework of Bowditch's taut loop length spectrum it is also possible to use the work of \cite{MOW}. 
Let $\mathcal{G}$ be the space of marked groups. 
Combining \cref{lem:trivial} and \cref{lem:minlength} we see that $\CZ\to \mathcal{G}$ given by $\sigma\mapsto G_L(\sigma)$ is a continuous injection with perfect image. 
Thus, by \cite[Theorem 1.1]{MOW} we obtain uncountably many quasi-isometry classes of groups of the form $G_L(\sigma)$. 
By carefully choosing subsets of $\CZ$ to ensure that the image is still perfect one can proves analogues of \cref{thm:uncountable} for other properties satisfied by the $G_L(\sigma)$.

\bibliographystyle{plain}
\bibliography{bib}

\end{document}